\documentclass[12pt]{article}
\usepackage[utf8]{inputenc}
\usepackage{graphicx}
\usepackage{amsmath}
\usepackage{fullpage}
\usepackage{amsfonts}
\usepackage{amssymb}
\usepackage{amsthm}
\newtheorem{thm}{Theorem}[section]
\newtheorem{lem}{Lemma}[section]

\newtheorem{dfn}{Definition}[section]
\newtheorem{cor}{Corollary}[section]
\newtheorem{rmk}{Remark}[section]

\title{ On the spectral radius and the energy of eccentricity matrix of a graph}

\author{Iswar Mahato\thanks{Department of Mathematics, Indian Institute of Technology Kharagpur, Kharagpur 721302, India. Email: iswarmahato02@gmail.com}\and R. Gurusamy\thanks{Department of Mathematics, Mepco Schlenk Engineering College, Sivakasi 626005, Tamil Nadu, India. Email: sahama2010@gmail.com}  \and M. Rajesh Kannan\thanks{Department of Mathematics, Indian Institute of Technology Kharagpur, Kharagpur 721302, India. Email: rajeshkannan@maths.iitkgp.ac.in, rajeshkannan1.m@gmail.com } \and S. Arockiaraj\thanks{Department of Mathematics, Government Arts and Science College, Sivakasi 626124, Tamil Nadu, India. Email: psarockiaraj@gmail.com}
}
\date{\today}
\begin{document}
\maketitle
\baselineskip=0.25in

\begin{abstract}

The eccentricity matrix $\varepsilon(G)$ of a graph $G$ is obtained from the distance matrix by retaining the eccentricities (the largest distance) in each row and each column.  In this paper, we give a characterization of the star graph, among the trees, in terms of invertibility of the associated eccentricity matrix. The largest eigenvalue of  $\varepsilon(G)$ is called the $\varepsilon$-spectral radius, and the eccentricity energy (or the $\varepsilon$-energy) of $G$ is the sum of the absolute values of the eigenvalues of $\varepsilon(G)$. We establish some bounds for the $\varepsilon$-spectral radius and characterize the extreme graphs. Two graphs are said to be $\varepsilon$-equienergetic if they have the same $\varepsilon$-energy. For any $n \geq 5$, we construct a pair of  $\varepsilon$-equienergetic graphs on $n$ vertices,  which are not $\varepsilon$-cospectral.
\end{abstract}


\section{Introduction}\label{sec1}

All graphs considered in this paper are finite and simple graphs, that is graphs without loops, multiple edges or directed edges. Let $G=(V(G),E(G))$ be a graph with  vertex set $V(G)=\{v_1,v_2,\hdots,v_n\}$ and  edge set $E(G)=\{e_1, \dots , e_m \}$. The \emph{adjacency matrix} of a graph $G$, denoted by $A(G) = (a_{uv})_{n\times n}$, is the $0-1$ matrix whose rows and columns are indexed by the vertices of $G$, and is defined by $a_{uv}=1$ if and only if the vertices $u$ and $v$  are adjacent, and $a_{uv}=0$ otherwise. For two vertices $u,v \in V(G)$, let $P(u,v)$ denote the path joining the vertices $u$ and $v$. The \emph{distance} between the vertices $u,v\in V(G)$, denoted by $d_{G}(u,v)$, is  the minimum length of the paths between $u$  and $v$. Let $D(G)=(d_{uv})_{n\times n}$ be the distance matrix of $G$, where $d_{uv}=d_{G}(u,v)$.  The \emph{eccentricity} $e(u)$ of the vertex $u$ is defined  as $e(u)=max \{d(u,v):  v \in V(G)\}$. A vertex $v$ is said to be an eccentric vertex of the vertex $u$ if $d_{G}(u,v)=e(u)$. The diameter $diam(G)$, and the radius $rad(G)$ of a graph $G$, is the maximum and the minimum  eccentricity of all  vertices of $G$, respectively. A vertex $u\in V(G)$ is said to be \emph{diametrical vertex} of $G$ if $e(u)=diam(G)$. If each vertex of $G$ has a unique diametrical vertex, then $G$ is called the \emph{diametrical graph} which is studied and referred to as even graphs in \cite{gobel1986even}.

The \emph{eccentricity matrix $\varepsilon(G)=(\epsilon_{uv})$} of a graph $G$, which is introduced in \cite{ran1, ecc-main} and further studied in \cite{ ours1, ran1, ecc-main}, is defined as
$$\epsilon_{uv}=
\begin{cases}
\text{$d_G(u,v)$} & \quad\text{if $d_G(u,v)=min\{e(u),e(v)\}$,}\\
\text{0} & \quad\text{otherwise.}
\end{cases}$$
In \cite{ran1,ran2}, the eccentricity matrix  is  known as $D_{\max}$-matrix. The eigenvalues of the eccentricity matrix of a graph $G$ is called the $\varepsilon$-eigenvalues of $G$. Since $\varepsilon(G)$ is symmetric, all of its eigenvalues are real. Let $\xi_1>\xi_2>\hdots >\xi_k$ be all the distinct $\varepsilon$-eigenvalues of $G$, then the $\varepsilon$-spectrum of $G$ can be written as
 \[ spec_{\varepsilon}(G)=
  \left\{ {\begin{array}{cccc}
   \xi_1 & \xi_2  &\hdots & \xi_k\\
   m_1& m_2& \hdots &m_k\\
  \end{array} } \right\},
  \]
 where $m_i$ be the algebraic multiplicity of $\xi_i$ for $i=1,2,\hdots,k$. The largest eigenvalue of $\varepsilon(G)$ is called the $\varepsilon$-spectral radius and is denoted by $\rho(\varepsilon(G))$.

 It is well-known that  \emph{graph energy} is a vital chemical index in chemical graph theory. The energy ( or $A$-energy ) of a graph is introduced in \cite{gutman-energy}, which is defined as
 $$E_{A}(G)=\sum_{i=1}^n |\lambda_i|,$$
 where $\lambda_i$, $i=1,2,\hdots,n$ are the eigenvalues of the adjacency matrix of $G$. In a similar way, the eccentricity energy (or $\varepsilon$-energy ) of a graph $G$ is defined \cite{wang2019graph} as
$$E_{\varepsilon}(G)=\sum_{i=1}^n |\xi_i|,$$
where $\xi_1,\xi_2,\hdots,\xi_n$ are the $\varepsilon$-eigenvalues of $G$. Two graphs are said to be \emph{$\varepsilon$-cospectral} if they have the same $\varepsilon$-spectrum, and two graphs are said to be  \emph{$\varepsilon$-equienergetic} if they have the same $\varepsilon$-energy. We are, of course,  interested in studying about
$\varepsilon$-equienergetic graphs which are not $\varepsilon$-cospectral.

The \emph{Wiener index} of a graph is an important and well studied topological index in mathematical chemistry. It is defined as
$$W(G)=\frac{1}{2}\sum_{{u,v}\in V(G)}d_{G}(u,v).$$
Similarly, we define the \emph{ eccentric Wiener index} (or $\varepsilon$-Wiener index) of a connected graph $G$ as follows
$$W_{\varepsilon}(G)=\frac{1}{2}\sum_{{u,v}\in V(G)}\epsilon_{uv}.$$
As usual let $K_{1,n-1}$, $P_n$ and $K_n$ denote the star, the path and the complete graph on $n$ vertices, respectively. For other undefined notations and terminology from graph theory, we refer to \cite{bon-mur-book}. We shall use the following results for the proof of our main results.

\begin{lem}\label{lem:comp}
Let $A=(a_{ij})$ and $B=(b_{ij})$ be two $n \times n$ matrices such that $b_{ij}\geq a_{ij}$ for all $i,j$. Then $\rho(B)\geq \rho(A)$.
\end{lem}

\begin{lem}\cite[Lemma 2.1]{lin2013distance} \label{lem:star}
The graph $K_{1,n-1}$ is the unique graph, which have maximum distance spectral radius among all graphs with diameter 2.
\end{lem}

\begin{thm}\cite{hor-john-mat}(Interlacing Theorem)
	Let $A$ be a  symmetric matrix of order $n$ and let $B$ be its principal submatrix of order $m<n$. Suppose  $\lambda_1(A) \leq \lambda_2(A) \hdots \leq \lambda_n(A)$are the eigenvalues of $A$ and  $\beta_1(B) \leq \beta_2(B)  \hdots \leq  \beta_m(B)$ are the eigenvalues of $B$. Then,
	$\lambda_i(A) \leq \beta_i(B) \leq  \lambda_{i+n-m}(A)$ for $i = 1,\hdots , m$,
	and if $ m = n-1$, then
	$\lambda_1(A)\leq \beta_1(B) \leq \lambda_2(A) \leq \beta_2(B) \leq \hdots \leq \beta_{n-1}(B) \leq \lambda_n(A)$.
\end{thm}

\begin{dfn}\cite{hor-john-mat}
 (Equitable partitions) Let $A$ be a real symmetric matrix whose rows and columns are indexed by $X=\{1,2,\hdots,n\}$. Let $\pi=\{X_1,X_2,\hdots,X_m\}$ be a partition of $X$. The characteristic matrix $C$ is the $n\times m$ matrix whose $j$-th column is the characteristic vector of $X_j$ $(j=1,2,\hdots,m)$. Let $A$ be partitioned according to $\pi$ as \[A=\left[ {\begin{array}{cccc}
		A_{11} & A_{12} &\hdots & A_{1m}\\
		A_{21} & A_{22} &\hdots & A_{2m}\\
		\vdots &\hdots & \ddots & \vdots\\
		A_{m1} & A_{m2}& \hdots &A_{mm}\\
		\end{array} } \right],\] 
where $A_{ij}$ denotes the submatrix (block)	 of $A$ formed by rows in $X_i$ and the columns in $X_j$. If $q_{ij}$ denote the average row sum of $A_{ij}$, then the matrix $Q=(q_{i,j})$ is called the quotient matrix of $A$. If the row sum of each block $A_{ij}$ is a constant, then the partition $\pi$ is called equitable partition.
\end{dfn}

\begin{thm}\cite{cvetkovic2009introduction}\label{quo-spec}
 Let $Q$ be a quotient matrix of any square matrix $A$ corresponding to an equitable partition. Then the spectrum of $A$ contains the spectrum of $Q$.
\end{thm}

This article is organized as follows: In section $2$, we show that the eccentricity matrix of a tree, other than $P_4$, is invertible if and only if it is the star. In section $3$, we obtain bounds for $\varepsilon$-spectral radius of graphs and characterize the extreme graphs. In section $4$, we construct a pair of non-cospectral $\varepsilon$-equienergetic graphs.

\section{A characterization of star graph}
In this section, we prove that among all trees,  other than $P_4$, star is the only graph for which the eccentricity matrix is always invertible.
\begin{thm}
Let $T$ be a tree, other than $P_4$, then the eccentricity matrix of $T$ is invertible if and only if $T$ is the star.
\end{thm}
\begin{proof}
Let $T$ be the star on $n$ vertices. As the distance matrix and the eccentricity matrix of the star are same, so $\det(\varepsilon(T)=(-1)^{n-1}(n-1)2^{n-2}$. Thus $\varepsilon(T)$ is invertible.

To prove the converse, first, let us consider the trees of order up to $4$. For $n = 2, 3$, the proof is trivial.  For $n=4$,  $P_4$ and $K_{1,3}$ are the only trees of order $4$, and the eccentricity matrix of both the trees are invertible.
	
Let $T$ be a tree on $n \geq 5$ vertices other than the star. We will show that $det(\varepsilon(T))=0$. Let $P(v_1,v_m)=v_1v_2\hdots v_{m-1}v_m$ be a diametrical path of length $m-1$ in $T$. Now consider the following two cases:\\
\textbf{Case(I)}: Let either $v_2$ or $v_{m-1}$ be adjacent to at least one pendant vertex other than the vertices $v_1$ and $v_m$. Without loss of generality, assume that $v_{m-1}$ is adjacent to $p$ pendant vertices, say, $u_1,u_2,\hdots,u_p$. Then the rows corresponding to the vertices $u_1,u_2,\hdots,u_p$ and $v_m$ are the same in $\varepsilon(T)$. Thus $det(\varepsilon(T))=0$.\\
\textbf{Case(II)}: Let both the vertices $v_2$ and $v_{m-1}$ are not  adjacent to any of the pendant vertices in $G$ other than $v_1$ and $v_m$, respectively. Since $T$ is a tree other than the star, so $diam(T)\geq 3$. If $T$ is a tree on $n\geq 5$ vertices and $diam(T)=3$, then one of the vertices $v_2$ or $v_{m-1}$ must be adjacent to at least  two pendent vertices, and the proof follows from case(I). Let $diam(T)\geq 4$. Let us show that at least two rows of $\varepsilon(T)$ are linearly dependent. Now we consider the following two subcases:\\
Subcase(I): Let $diam(T)=4$, and let  $P(v_1,v_5)=v_1v_2v_3 v_4v_5$ be a diametrical path in $T$.  Let $u_1,u_2,\hdots,u_p$ be the vertices, other than $v_1$ and $v_5$, such that each $u_i$ has exactly one common neighbour, say  $w_i$, with $v_3$. It is easy to see that, the vertices  $u_1,u_2,\hdots,u_p$ are pendant.  The rows corresponding to the vertices $w_1,w_2,\hdots,w_p,v_2,v_4$ and the row corresponding to the vertex $v_3$, in $\varepsilon(T)$, are linearly dependent .\\
Subcase(II): Let $diam(T)\geq 5$, and let $P(v_1,v_m)=v_1v_2v_3\hdots v_{m-1}v_m$ be a diametrical path in $T$. Then the rows corresponding to the vertices $v_2$ and $v_3$ are linearly dependent in  $\varepsilon(T)$.

Thus $\det(\varepsilon(T)) =0$ in all the above cases. Therefore, if the eccentricity matrix of $T$ is invertible, then $T$ is the star.
\end{proof}

\section{Bounds for $\varepsilon$-spectral radius of graphs}
In this section, we  establish bounds for the $\varepsilon$-spectral radius of graphs, and characterize the extreme graphs. In the next theorem,  we derive a characterization for the star, among all connected graphs with diameter $2$, in terms of the $\varepsilon$-spectral radius.
\begin{thm}
Among all connected  graphs on $n$ vertices with diameter $2$, the star $K_{1,n-1}$ is the unique graph, which has maximum $\varepsilon$-spectral radius.
\end{thm}
\begin{proof} Let $G$ be a connected graph on $n$ vertices such that $diam(G)=2$. From the definition, it follows that the eccentricity matrix  $\varepsilon(G)$ of $G$ is entrywise dominated by the distance matrix $D(G)$. So by Lemma \ref{lem:comp}, $\rho(\varepsilon(G))\leq \rho(D(G))$. For $K_{1,n-1}$, the star on $n$ vertices,  the eccentricity matrix and the distance matrix are the same, and hence $\rho(D(K_{1,n-1}))=\rho(\varepsilon(K_{1,n-1}))$. By Lemma \ref{lem:star},  $\rho(D(G))\leq \rho(D(K_{1,n-1}))=(n-2)+\sqrt{n^2-3n+3}$, and the equality holds if and only if $G$ is the star. Therefore, $\rho(\varepsilon(G))\leq \rho(D(G))\leq \rho(D(K_{1,n-1}))=\rho(\varepsilon(K_{1,n-1}))$, and the equality holds only for the star.
\end{proof}

Next we establish an lower bound for the $\varepsilon$-spectral radius of a graph with given diameter, and  characterize the extreme graph.
\begin{thm}\label{prop:diam}
If $G$ is a connected graph with diameter $d\geq 2$, then $\rho(\varepsilon(G))\geq d$, and the equality holds if and only if $G$ is the diametrical graph with diameter $d$.
\end{thm}
\begin{proof}
 Let $G$  be a connected graph with diameter $d\geq 2$. Then there exists a $2 \times 2$ principal submatrix
 $ \left [ {\begin{array}{cc}
 0 & d\\
 d & 0\\
\end{array} } \right ]$
 whose eigenvalues are $ d,-d$. Thus, by interlacing theorem, we have  $\rho(\varepsilon(G))\geq d$.

Let $G$ be a diametrical graph with diameter $d$. Then for each vertex $v$ of $G$, the eccentricity $e(v)=diam(G)=d$, and the eccentricity attains for a unique vertex. So the eccentricity matrix of $G$ can be written as   $ \left [ {\begin{array}{cc}
 0 & dI_k\\
 dI_k & 0\\
\end{array} } \right ]$, whose $\varepsilon$-spectrum is  $ \left \{ {\begin{array}{cc}
 d & -d\\
 k & k\\
\end{array} } \right \}$. Thus $\rho(\varepsilon(G))=d$.

Conversely, let $\rho(\varepsilon(G))=d$. Suppose $G$ is not the diametrical graph. Then we have the following cases:\\
\textbf{Case(I):} Let $G$ be a graph such that $rad(G)=diam(G)=d$. Then $$B = \left[ {\begin{array}{ccc}
0 & d & d\\
d & 0 & 0\\
d & 0 & 0 \\
\end{array} } \right]$$ is a principal submatrix of $\varepsilon(G)$, and $\rho(B) = d\sqrt{2}$. Therefore, by interlacing theorem, we have $\rho(\varepsilon(G)) \geq d \sqrt{2}>d$, which is not possible.\\
\textbf{Case(II):} Let $G$ be a graph such that $rad(G) \neq diam(G)=d$. Then there exists a vertex $v_k$ with eccentricity $e(v_k)=k<d$. Let $v_1$ be a vertex of $G$ with $e(v_1)=d$. Since $G$ is a connected graph, there is a path $P(v_1,v_k)$ between the vertices $v_1$ and $v_k$. It is easy to see that the eccentricity of any vertex which is adjacent to $v_1$ is either $d$ or $d-1$. Hence, in  the path $P(v_1,v_k)$ there always exists a pair of adjacent vertices $u$ and $v$ such that $e(u)=d$ and $e(v)=d-1$. Let $w$ be the eccentric vertex of $u$, that is, $d(u,w)=d$. Then $d(v,w)=d-1$. Since $e(v)=d-1$ and $w$ is an eccentric vertex of $v$, the $vw$-th entry of $\varepsilon(G)$ is $d-1$. Therefore,
$$C = \left[ {\begin{array}{ccc}
0 & 0 & d\\
0 & 0 & d-1\\
d & d-1 & 0 \\
\end{array} } \right]$$ is a principal submatrix of $\varepsilon(G)$, corresponding to the vertices $u,v$ and $w$. Now, since $\rho(C)$ equals to $\sqrt{(d-1)^2+d^2}$, by interlacing theorem, we have $\rho(\varepsilon(G)) \geq \sqrt{(d-1)^2+d^2}>d$, which is a contradiction.

Therefore, $G$ is a diametrical graph. This completes the proof.
\end{proof}

\begin{cor}
Among the connected bipartite graphs on $2n$ $(n\geq 3)$ vertices, the graph $W_{n,n}$ has the minimum $\varepsilon$-spectral radius, where  $W_{n,n}$ is the graph obtained by deleting $n$ independent edges from the complete bipartite graph $K_{n,n}$.
\end{cor}
\begin{proof}
Since  $W_{n,n}$ is obtained by deleting $n$ independent edges from $K_{n,n}$, each vertex of $W_{n,n}$ has a unique diametrical vertex with eccentricity 3. Therefore, $W_{n,n}$ is a diametrical graph with diameter 3. So, by Theorem \ref{prop:diam}, we have $\rho(\varepsilon(W_{n,n}))=3$.

Among the bipartite graphs on $2n$ vertices, $K_{1,2n-1}$ and $K_{n,n}$ are the only graphs of diameter 2 and $\rho(\varepsilon(K_{1,2n-1}))=2(n-1)+\sqrt{4n^2-6n+3} \geq 3$, $\rho(\varepsilon(K_{n,n}))=2(n-1)\geq 3$. Therefore, the proof follows from Theorem \ref{prop:diam}.
\end{proof}
The $\varepsilon$-degree of a vertex $v_i\in V(G)$ is defined as $\varepsilon(i)=\sum_{j=1}^n\epsilon_{ij}$. A graph $G$ is said to be $\varepsilon$-regular if $\varepsilon(i)=k$ for all $i$ \cite{wang2019graph}. Now let us establish a lower bound for the $\varepsilon$-spectral radius of a graph in terms of eccentric Wiener index.

\begin{thm}\label{bou-wie}
Let $G$ be a connected graph on $n$ vertices with eccentric Wiener index $W_{\varepsilon}$. Then $\rho(\varepsilon(G))\geq \frac{2W_{\varepsilon}}{n}$ and the equality holds if and only if $G$ is $ \varepsilon$-regular graph.
\end{thm}
\begin{proof}
Let $x=\frac{1}{\sqrt{n}}[1,1,\hdots,1]^T$ be the unit positive vector of order $n$.  By applying Rayleigh Principle to the eccentricity matrix $\varepsilon(G)$ of the graph $G$, we get
\begin{eqnarray*}
\rho(\varepsilon(G)) \geq \frac{x^T\varepsilon(G)x}{x^Tx} &=&\frac{1}{\sqrt{n}}[1,1,\hdots,1] \frac{1}{\sqrt{n}}[\varepsilon(1),\varepsilon(2),\hdots,\varepsilon(n)]^T\\
&=& \frac{1}{n}\sum_{i=1}^n \varepsilon(i)\\
&=& \frac{2W_{\varepsilon}}{n}.
\end{eqnarray*}
Now, if $G$ is $\varepsilon$-regular, then each row sum of $\varepsilon(G)$ is a constant, say $k$ and hence $\rho(\varepsilon(G))=k$. Therefore, $\rho(\varepsilon(G))=k=\frac{nk}{n}= \frac{2W_{\varepsilon}}{n}$, and hence the equality holds.

Conversely if equality holds, then $x$ is an eigenvector corresponding to $\rho(\varepsilon(G))$ and hence $\varepsilon(G)x=\rho(\varepsilon(G))x$. Therefore, $\varepsilon(i)=\rho(\varepsilon(G))$ for all $i$. Thus $G$ is $\varepsilon$-regular. This completes the proof.
\end{proof}

\begin{cor}
Let $G$ be a connected graph on $n$ vertices and $m$ edges with diameter $2$. If $G$ has $k$ vertices of degree $n-1$, then
\begin{equation}\label{eq1}
\rho(\varepsilon(G)) \geq \frac{2(n^2-n-2m)+k(2n-k-1)}{n}.
\end{equation}
\end{cor}
\begin{proof}
Let $G$ be a connected graph of diameter 2 with vertex set $\{v_1,v_2,\hdots,v_k,v_{k+1},\hdots,v_n\}$, where $v_1,v_2,\hdots,v_k$ are the vertices of degree $n-1$. Therefore, $e(v_i)=1$ for $i=1,2,\hdots,k$ and $e(v_i)=2$ for $i=k+1,\hdots,n$. Then
\begin{eqnarray*}
2W_{\varepsilon}(G)=\sum_{i=1}^n \varepsilon(i) &=& k(n-1)+  \sum_{i=k+1}^n \Big( k+2\big((n-k)-(d_i-k)-1\big)\Big)\\
&=&k(n-1)+  \sum_{i=k+1}^n \big( k+2(n-d_i-1)\big)\\
&=& 2(n^2-n-2m)+k(2n-k-1).
\end{eqnarray*}
Thus the proof follows from Theorem \ref{bou-wie}.
\end{proof}
In the next result, we obtain a lower bound for the $\varepsilon$-spectral radius in terms of $\varepsilon$-Wiener index, and $\varepsilon$-degree sequence. 
\begin{thm}\label{wie-deg}
Let $G$ be a connected graph of order $n$ with $\varepsilon$-Wiener index $W_{\varepsilon}$ and $\varepsilon$-degree sequence $\{\varepsilon(1),\varepsilon(2),\hdots,\varepsilon(n)\}$. Then
$$\rho(\varepsilon(G))\geq \max_{i}\Big\{\frac{1}{n-1}\Big(\big(W_{\varepsilon}-\varepsilon(i)\big)+\sqrt{\big(W_{\varepsilon}-\varepsilon(i)\big)^2+(n-1){\varepsilon}^2(i)}\Big)\Big\}.$$
\end{thm}
\begin{proof}
Let $v_i$ be a vertex of the graph $G$ and $\varepsilon(i)$ be its $\varepsilon$-degree. Let us partition the eccentricity matrix of $G$ with respect to the row corresponding to the vertex $v_i$. Then the quotient matrix corresponding to this partition is
\[ A=\left[ {\begin{array}{cc}
		0 & \varepsilon(i)\\
	   \frac{\varepsilon(i)}{n-1} & \frac{2\big(W_{\varepsilon}-\varepsilon(i)\big)}{n-1} \\
\end{array} } \right].\]
The eigenvalues of $A$ are
$$\mu_1=\frac{1}{n-1}\Big\{\big(W_{\varepsilon}-\varepsilon(i)\big)+\sqrt{\big(W_{\varepsilon}-\varepsilon(i)\big)^2+(n-1){\varepsilon}^2(i)}\Big\}$$ and
$$\mu_2=\frac{1}{n-1}\Big\{\big(W_{\varepsilon}-\varepsilon(i)\big)-\sqrt{\big(W_{\varepsilon}-\varepsilon(i)\big)^2+(n-1){\varepsilon}^2(i)}\Big\}.$$
From Lemma \ref{quo-spec}, we have
$$\rho(\varepsilon(G))\geq \mu_1=\frac{1}{n-1}\Big\{\big(W_{\varepsilon}-\varepsilon(i)\big)+\sqrt{\big(W_{\varepsilon}-\varepsilon(i)\big)^2+(n-1){\varepsilon}^2(i)}\Big\}.$$
Since this is true for all $i$, the proof is done.
\end{proof}
\begin{rmk}
	The counterparts of the bounds provided in Theorem \ref{bou-wie} and Theorem  \ref{wie-deg} for the distance matrix case is known in the literature \cite{indu-laa}.
	
	\end{rmk}
\section{Construction of $\varepsilon$-equienergetic graphs}
The problem of constructing non-cospectral equienergetic graphs is  an interesting problem in spectral graph theory. Motivated by this, in this section we discuss  the construction of $\varepsilon$-equienergetic graphs.
\begin{lem}\label{bi-ener}
Let  $K_{p,q}$ be a  complete bipartite graph on $n=p+q$ vertices. If $p, q \geq 2$, then the $\varepsilon$-energy of $K_{p,q}$ is $4(p+q-2)$.
\end{lem}
\begin{proof} The eccentricity matrix of $K_{p,q}$ can be written as \[ \varepsilon(K_{p,q})=\left[ {\begin{array}{cc}
		2(J_p-I_p) & 0\\
		0 & 2(J_q-I_q) \\
		\end{array} } \right]\]
Therefore, \begin{equation}\label{bi-spec} spec_{\varepsilon}(K_{p,q})=\left\{ {\begin{array}{ccc}
		2(p-1) & 2(q-1)& -2\\
		1 & 1& p+q-2 \\
		\end{array} } \right\},
\end{equation}		
and hence $E_{\varepsilon}(K_{p,q})=4(p+q-2)$.
\end{proof}
There are only two connected graphs of order $3$, and only six connected graphs of order $4$. By an elementary calculation, we can say that there does not exist any $\varepsilon$-equienergetic graphs of order 3 and 4.
Let us consider the graphs $G_1$ and $G_2$ of order $5$ as shown in Figure $1$.
\begin{figure}
\centering
\includegraphics[height=4.5 cm, width=10.5 cm]{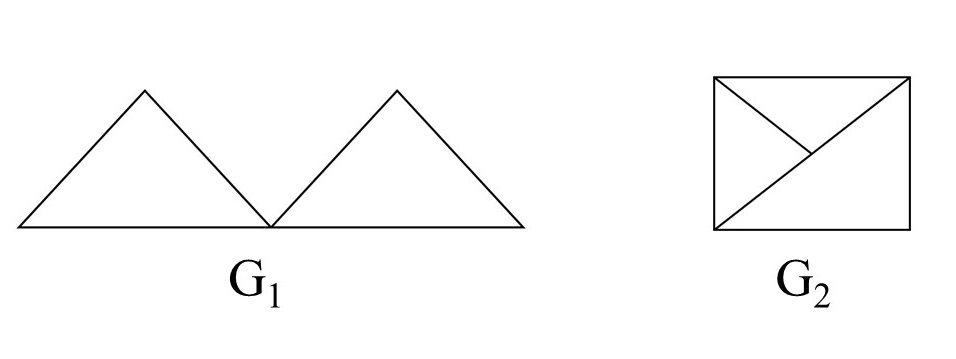}
\caption{Non-cospectral $\varepsilon$-equienergetic graphs of order $5$}
\end{figure}
Therefore,
\begin{align*}
spec_{\varepsilon}(G_1)=\left\{ {\begin{array}{ccc}
		2(1\pm \sqrt{2}) & -4 & 0\\
		1  & 1 & 2 \\
		\end{array} } \right\} \qquad \mbox{and} \qquad
spec_{\varepsilon}(G_1)=\left\{ {\begin{array}{ccc}
		\pm 2\sqrt{2}) & \pm 2& 0\\
		1 & 1 & 1  \\
		\end{array} } \right\}.		
\end{align*}
Hence, $E_{\varepsilon}(G_1)=4+4\sqrt{2}=E_{\varepsilon}(G_2)$. So, the graphs $G_1$ and $G_2$ are non-$\varepsilon$-cospectral $\varepsilon$-equienergetic graphs of order $5$. In the next theorem we show that, for $n\geq 6$, there exists a pair of non-$\varepsilon$-cospectral graphs which are $\varepsilon$-equienergetic.
\begin{thm}
For $n\geq 6$ and $p,q \geq 2$, the graphs $K_{p,n-p}$ and $K_{q,n-q}$ are $\varepsilon$-equienergetic, but not $\varepsilon$-cospectral  .
\end{thm}
\begin{proof}Proof follows from Lemma \ref{bi-ener}.
\end{proof}

Also, we have more general result. If $K_{n_1,n_2,\hdots,n_k}$ is a complete $k$-partite graphs on $n$ vertices, where $n=\sum_{i=1}^nn_i$ with $n_i\geq 2$, then by Theorem $4.6$ in \cite{ ours1}, we have
\[spec_{\epsilon}(K_{n_1,\hdots,n_k})=
	\left\{ {\begin{array}{ccccc}
		-2 & 2(n_1-1) & 2(n_2-1)&\hdots & 2(n_k-1)\\
		n-k &1 &1&\hdots &1\\
		\end{array} } \right\}.\]
Thus, $E_{\varepsilon}(K_{n_1,n_2,\hdots,n_k})=4(n-k)$, which is independent of $n_1,n_2,\hdots,n_k$. Therefore, every complete $k$-partite graphs are  $\varepsilon$-equienergetic but not $\varepsilon$-cospectral if they have at least $2$ vertices in each partition.

\textbf{Acknowledgement:} Iswar Mahato and  M. Rajesh Kannan would like to thank  the Department of Science and Technology, India, for financial support through the Early Career Research Award (ECR/2017/000643).

\bibliographystyle{plain}
\bibliography{reference}
\end{document}